\documentclass{amsart}
\calclayout
\usepackage{amssymb,amsmath,amsfonts,epsfig,latexsym,tikz}
\usepackage[alphabetic]{amsrefs}
\usepackage{tikz-cd}
\usepackage{hyperref}
\usepackage{enumerate}
\usepackage{mathtools}
\usepackage{verbatim}
\usepackage{cleveref}
\usepackage[shortlabels]{enumitem}
\usepackage{subcaption}
\usepackage{ mathrsfs }

\usetikzlibrary{positioning}
\usetikzlibrary{matrix}
\usetikzlibrary{decorations}
\usetikzlibrary{decorations.pathreplacing, decorations.pathmorphing, angles,quotes}

\newtheorem{theorem}{Theorem}[section]
\newtheorem{definition}[theorem]{Definition}
\newtheorem{proposition}[theorem]{Proposition}
\newtheorem{lemma}[theorem]{Lemma}
\newtheorem{corollary}[theorem]{Corollary}
\newtheorem{conjecture}[theorem]{Conjecture}

\theoremstyle{definition}
\newtheorem{remark}[theorem]{Remark}
\newtheorem{example}[theorem]{Example}

\def\R{\mathbb{R}}

\def\Z{\mathbb{Z}}

\DeclareMathOperator{\Int}{Int}

\DeclareMathOperator{\hht}{ht}

\begin{document}

\title{Lefschetz properties of local face modules}          

\author{Matt Larson and Alan Stapledon}
\date{\today}
\begin{abstract}
Local face modules are modules over face rings whose Hilbert function is the local $h$-vector of a triangulation of a simplex. We study when Lefschetz properties hold for local face modules. We prove new inequalities for local $h$-vectors of vertex-induced triangulations by proving Lefschetz properties for local face modules of these triangulations. We show that, even for regular triangulations, Lefschetz properties can fail for local face modules in positive characteristic. 
\end{abstract}

\maketitle

\vspace{-20 pt}

\section{Introduction}

We study an algebraic object associated to a triangulation of a simplex which was introduced by Stanley \cite{Stanleylocalh} in order to prove results about \emph{local $h$-vectors}, vectors that control how the $h$-vector of a simplicial complex changes under subdivision. 
Local $h$-vectors play an important role in the study of the decomposition theorem for proper toric morphisms \cite{deCataldoMiglioriniMustata18,KatzStapledon16} and in the study of the monodromy action on the cohomology of the Milnor fiber of a Newton nondegenerate hypersurface \cite{Stapledon17,LPS2}.
See \cite{ChanSurvey,AthanasiadisSurvey} for surveys  on local $h$-vectors. 

The properties of local $h$-vectors are heavily dependent on the precise notion of triangulation which is used. The prototypical example of a triangulation of a simplex is a \emph{geometric triangulation}, which is a triangulation of the geometric realization of a simplex $2^V$ into geometric simplices, where $V = \{ 1,\ldots, d\}$ throughout. 
Associated to a geometric triangulation is a map $\sigma \colon \Gamma \to 2^V$ from the abstract simplicial complex associated to the triangulation to the simplex, which takes a face of $\Gamma$ to the smallest face of $2^V$ containing it. 

We will consider much more general notions of triangulations of simplices. 
A \emph{topological triangulation} of a simplex is an order-preserving  map $\sigma \colon \Gamma \to 2^V$ of posets from an abstract simplicial complex $\Gamma$ to a simplex $2^V$, such that, for every subset $U \subset V$, 
the geometric realization of the subcomplex $\Gamma_U := \sigma^{-1}(2^U)$ of $\Gamma$ is homeomorphic to a $(|U| - 1)$-dimensional ball, and $\sigma^{-1}(U)$ is the set of interior faces of the ball $\Gamma_U$. 
A face $F$ of $\Gamma$ is \emph{interior} if $\sigma(F) = V$. 
See Section~\ref{sec:triang} for a discussion of triangulations of simplices, including the definitions of regular, geometric, vertex-induced, and quasi-geometric triangulations.

The {local $h$-vector} $\ell(\Gamma) = (\ell_0(\Gamma), \dotsc, \ell_d(\Gamma))$ of a topological triangulation $\sigma \colon \Gamma \to 2^V$ of a $(d-1)$-dimensional simplex is defined as an alternating sum of the $h$-vectors of inverse images of simplices in $2^V$, see \eqref{eq:localh}. 
Let $f_i^j$ be the number of faces $G$ of $\Gamma$ with $|G| = i$ and $|\sigma(G)| = j$. Then $\ell_0(\Gamma) = 0$, $\ell_1(\Gamma) = f_1^d$ is the number of interior vertices, and $\ell_2(\Gamma) = f_2^d - f_1^{d-1} - (d-1)f_1^{d}$. See \cite[Example 2.3]{Stanleylocalh} and \eqref{eq:excessformula}. In his paper introducing local $h$-vectors, Stanley established the following fundamental properties. 

\begin{theorem}\cite{Stanleylocalh}\label{thm:localhstructure}
Let $\sigma \colon \Gamma \to 2^V$ be a topological triangulation of a $(d-1)$-dimensional simplex. Then:
\begin{enumerate}
\item The local $h$-vector is symmetric, i.e., $\ell_s(\Gamma) = \ell_{d-s}(\Gamma)$ for all $s$. We have $\ell_0(\Gamma) = 0$ and $\ell_1(\Gamma) \ge 0$.
\item If $\Gamma$ is quasi-geometric, then the local $h$-vector is nonnegative, i.e., $\ell_s(\Gamma) \ge 0$ for all $s$. 
\item If $\Gamma$ is regular, then $\ell(\Gamma)$ is unimodal, i.e., $\ell_0(\Gamma) \le \dotsb \le \ell_{\lfloor d/2 \rfloor}(\Gamma)$. 
\end{enumerate}
\end{theorem}

In his paper, Stanley introduced the combinatorial theory of posets now known as Kazhdan--Lustzig--Stanley theory and deduced (1) as a special case.
The proof of (2) uses an algebraic construction which will be central to the results of this paper, and the proof of (3) is obtained by applying the decomposition theorem and the relative Hard Lefschetz theorem \cite{BBDG} to a projective toric morphism. 
The following theorem gives a sort of converse to Theorem~\ref{thm:localhstructure}. 

\begin{theorem}\cite{ChanLocal,JKSS}\label{thm:construction}
Let $\ell = (\ell_0, \dotsc, \ell_d)$ be a vector in $\mathbb{Z}^{d+1}$. Then:
\begin{enumerate}
\item If $\ell$ is symmetric, $\ell_0 = 0$, and $\ell_1 \ge 0$, then there is a topological triangulation $\sigma \colon \Gamma \to 2^V$ of a $(d-1)$-dimensional simplex with $\ell = \ell(\Gamma)$. 
\item If furthermore $\ell$ is nonnegative, then $\Gamma$ can be taken to be quasi-geometric. 
\item If furthermore $\ell$ is unimodal, then $\Gamma$ can be taken to be regular. 
\end{enumerate}
\end{theorem}

Taken together, Theorems~\ref{thm:localhstructure} and \ref{thm:construction} characterize the possible local $h$-vectors of topological triangulations, quasi-geometric triangulations, and regular triangulations. 
Stanley had originally conjectured that the unimodality of local $h$-vectors could be extended to all quasi-geometric subdivisions \cite[Conjecture 5.4]{Stanleylocalh}, but a counterexample 
was found in \cite[Example 3.4]{AthanasiadisFlag}. Athanasiadis asked if the unimodality still holds for vertex-induced triangulations of simplices \cite[Question 3.5]{AthanasiadisFlag}. We give the first result towards this conjecture. 

\begin{theorem}\label{thm:combineq}
Let $\sigma \colon \Gamma \to 2^V$ be a vertex-induced triangulation of a $(d-1)$-dimensional simplex with $d \ge 3$. Then $\ell_1(\Gamma) \le \ell_2(\Gamma)$. 
\end{theorem}

In particular, if $d \le 5$, then Theorem~\ref{thm:combineq} and the symmetry of the local $h$-vector imply that the local $h$-vector is unimodal. Together with the construction in \cite{ChanLocal}, we have the following classification of the local $h$-vectors of vertex-induced triangulations. 

\begin{corollary}
Let $\ell = (\ell_0, \dotsc, \ell_d)$ be a vector in $\mathbb{Z}^{d+1}$ for some $d \le 5$. Then $\ell$ is the local $h$-vector of a vertex-induced triangulation of a $(d-1)$-dimensional simplex if and only if $\ell_0 = 0$, and $\ell$ is symmetric, nonnegative, and unimodal. 
\end{corollary}

The same result holds for geometric triangulations. The proof of Theorem~\ref{thm:combineq} is based on a Lefschetz result for \emph{local face modules}, graded modules associated to quasi-geometric triangulations whose Hilbert function is the local $h$-vector of the triangulation. We now recall the definition of local face modules.

Let $\sigma \colon \Gamma \to 2^V$ be a quasi-geometric triangulation of a $(d-1)$-dimensional simplex. 
Let $k$ be a field, and let $k[\Gamma]$ be the \emph{face ring} (or \emph{Stanley--Reisner ring}) of $\Gamma$. Let $k[\Gamma]^s$ be the degree $s$ part of $k[\Gamma]$. 
A linear system of parameters (l.s.o.p.) 
$ \theta = \{ \theta_1,\ldots,\theta_d \}$ with $\theta_i \in k[\Gamma]^1$
is \emph{special} if $\theta_i$ is supported on vertices $j$ of $\Gamma$ with $i \in \sigma(j)$. 
Let 
$\theta = \{ \theta_1,\ldots,\theta_d \}$ be a \emph{special l.s.o.p.} for $k[\Gamma]$.   Let $(\operatorname{int} \Gamma) = (x^G : \sigma(G) = V)$ be the ideal of interior faces in $k[\Gamma]$, and  let $A_{\theta}(\Gamma) = k[\Gamma]/(\theta_1, \dotsc, \theta_d)$.

\begin{definition}\label{def:localfacemodule}
For a special l.s.o.p. 
$\theta = \{ \theta_1,\ldots,\theta_d \}$,
the local face module $L_{\theta}(\Gamma)$ is the image of $(\operatorname{int} \Gamma)$ in
$A_{\theta}(\Gamma)$. 
\end{definition}

The local face module is a module over $k[\Gamma]$. If $k$ is infinite, then a triangulation is quasi-geometric if and only there is a special l.s.o.p. for $k[\Gamma]$, see \cite[Corollary 4.4]{Stanleylocalh}. The local face module inherits a grading $L_{\theta}(\Gamma) = \oplus_{s = 0}^d  L^s_{\theta}(\Gamma)$ 
from $k[\Gamma]$. 
In \cite[Theorem 4.6]{Stanleylocalh}, Stanley showed that, for any special l.s.o.p., 
$$\dim L^s_{\theta}(\Gamma) = \ell_s(\Gamma).$$
This proves the nonnegativity of the local $h$-vector of a quasi-geometric triangulation. 

We will usually work with \emph{generic} special l.s.o.p.s, i.e., l.s.o.p.s where the nonzero coefficients are algebraically independent. Let $K = k(a_{i,j})_{1 \le j \le n, \, i \in \sigma(j)}$, and set $\theta_i^{\operatorname{gen}} = \sum_{\sigma(j) \ni i} a_{i,j} x_j \in K[\Gamma]^1$. Let $L(\Gamma) := L_{\theta^{\operatorname{gen}}}(\Gamma)$ be the local face module constructed using a generic special l.s.o.p. Note that $L(\Gamma)$ depends on the choice of a field $k$. Set $\ell = x_1 + \dotsb + x_n \in K[\Gamma]^1$. 
Theorem~\ref{thm:combineq} is an immediate consequence of the following Lefschetz theorem for local face modules. 

\begin{theorem}\label{thm:WL}
Let $\sigma \colon \Gamma \to 2^V$ be a vertex-induced triangulation of a $(d-1)$-dimensional simplex with $d \ge 3$. If $k$ has characteristic $2$ or $0$, then multiplication by $\ell$ induces an injection from $L^1(\Gamma)$ to $L^2(\Gamma)$. 
\end{theorem}

We say that $L(\Gamma)$ has the \emph{strong Lefschetz property} if there is $u \in K[\Gamma]^1$ such that, for each $s \le d/2$, multiplication by $u^{d- 2s}$ induces an isomorphism from $L^s(\Gamma)$ to $L^{d-s}(\Gamma)$. A standard argument (see, e.g., \cite[Lemma 5.1]{LNS}) shows that $L(\Gamma)$ has the strong Lefschetz property if and only if multiplication by $\ell^{d- 2s}$ induces an isomorphism from $L^s(\Gamma)$ to $L^{d-s}(\Gamma)$ for all $s \le d/2$. 
Equivalently, $L(\Gamma)$ has the strong Lefschetz property if and only if the \emph{Hodge--Riemann form} on $L^s(\Gamma)$ given by $(u, v) \mapsto B(u, \ell^{d - 2s} \cdot v)$ is nondegenerate  for all $s \le d/2$. Here $B$ is a natural symmetric bilinear form on $L(\Gamma)$ first described by Stanley 
\cite[Corollary 4.19]{Stanleylocalh}; see Section~\ref{ssec:bilinear}. 
Recall that a bilinear form $b$ is \emph{anisotropic} if $b(x,x)$ is nonzero whenever $x$ is nonzero. This implies that $b$ is nondegenerate.

\begin{theorem}\label{thm:SL}
Let $\sigma \colon \Gamma \to 2^V$ be a vertex-induced triangulation of a $(d-1)$-dimensional simplex with $d \le 4$. If $k$ has characteristic $2$ or $0$, then for any $s \le d/2$, the Hodge--Riemann form on $L^s(\Gamma)$ is anisotropic. In particular, $L(\Gamma)$ has the strong Lefschetz property. 

\end{theorem}

In contrast to  Theorems~\ref{thm:WL} and \ref{thm:SL}, 
we provide 
counterexamples to Lefschetz properties of local face modules in positive characteristic. We say that $L(\Gamma)$ has the \emph{weak Lefschetz property} if, for any $s$, there exists $u \in K[\Gamma]^1$ such that the map
$L^s(\Gamma) \to L^{s+1}(\Gamma)$ induced by multiplication by $u$ has full rank. 
In Example~\ref{ex:counterexamples} we construct a family of regular  triangulations $\{\Gamma_{t}\}_{t \in \mathbb{N}}$ of a $(3t-1)$-dimensional simplex for which, for any special l.s.o.p. $\theta$,
\begin{itemize}
		\item For any $p > 0$, the strong Lefschetz property for $L_{\theta}(\Gamma_p)$ fails in characteristic $p$.
	\item For any $p > 0$, the weak Lefschetz property for $L_{\theta}(\Gamma_{2p - 1})$ fails in characteristic $p$. 
	\item For any $p$ (including $p = 0$) and $t > 1$, 
	the Hodge--Riemann form on $L^{t + 1}(\Gamma_t)_{\theta}$ is not anisotropic. 
\end{itemize}
We construct our example by
using the fact that Lefschetz properties are not preserved by tensor products in positive characteristic.

Our proofs of Theorem~\ref{thm:WL} and \ref{thm:SL} are based on the differential operator technique of Papadakis and Petrotou \cite{PapadakisPetrotoug}, which relies on certain identities that hold only in characteristic $2$, and also relies on the use of an l.s.o.p. that is `quite' generic. This technique has had many applications over the last few years \cite{APP,APPS,KaruXiao,Oba,LNS,KLS,moment}.
In particular, it has been used to show that analogues of Theorems~\ref{thm:WL} and \ref{thm:SL} hold for face rings of triangulations of spheres without any restriction on the dimension. 
On the other hand, for applications to local face modules,  the ``special'' condition forces the l.s.o.p. used to define $L(\Gamma)$ to be highly non-generic, making  the use of the differential operator technique difficult. In fact, the counterexamples above 
give a limitation on this technique. 
In the setting of Ehrhart theory, similar counterexamples obstruct attempts to prove the unimodality of 
local $h^*$-vectors
of certain lattice polytopes using the differential operators technique, see Example~\ref{ex:ehrhart}. Despite these counterexamples, we conjecture that the strong Lefschetz property holds for local face modules of vertex-induced triangulations over a field of characteristic $0$ if a generic special l.s.o.p. is used.

\begin{conjecture}\label{conj:char0}
Let $\sigma \colon \Gamma \to 2^V$ be a vertex-induced triangulation of a simplex. If $k$ has characteristic $0$, then $L(\Gamma)$ has the strong Lefschetz property. 
\end{conjecture}

We prove Conjecture~\ref{conj:char0} for regular triangulations. 

\begin{theorem}\label{thm:regular}
Let $\sigma \colon \Gamma \to 2^V$ be a regular triangulation of a simplex. If $k$ has characteristic $0$, then $L(\Gamma)$ has the strong Lefschetz property. 
\end{theorem}

Like Stanley's proof of the unimodality of local $h$-vectors for regular triangulations, our proof is based on the relative Hard Lefschetz theorem. We give a direct interpretation of the local face module in terms of the cohomology of a toric variety and then deduce the strong Lefschetz property from the relative Hard Lefschetz theorem. This extends Stanley's approach.

Another difficulty which arises when attempting to 
prove Lefschetz theorems for local face modules is that it is not clear what the minimal generators of the local face module are; it is not even obvious that $L(\Gamma)$ is generated as a $k[\Gamma]$-module in degree at most $d/2$. This makes it difficult to prove that a class in $L(\Gamma)$ vanishes using the nondegeneracy of the bilinear form. See Section~\ref{ssec:generation} for a discussion.

\begin{remark}
All of our results hold for an even more general class of triangulations, \emph{homology triangulations}, where the inverse images of simplices in $2^V$ are only required to be homology balls instead of being homeomorphic to a ball. See Remark~\ref{rem:homology1} and Remark~\ref{rem:homologylefschetz}.
\end{remark}

\emph{Relative local $h$-vectors} are invariants of triangulations $\sigma \colon \Gamma \to 2^V$ which also take into account a face $E$ of $\Gamma$. For quasi-geometric triangulations, there are generalizations of local face modules whose Hilbert functions are relative local $h$-vectors \cite{Athanasiadis12b,LPS1}.
Many of our results admit straightforward generalizations to relative local $h$-vectors, see Section~\ref{ssec:relative}.

\subsection*{Acknowledgements}
We thank Kalle Karu for explaining Stanley's nondegenerate bilinear form on local face modules to us.  We thank Christos Athanasiadis, Satoshi Murai, Isabella Novik, Ryoshun Oba, and Sam Payne for useful conversations. We thank the referee for their careful reading and helpful comments.  
This work was mostly conducted while the authors were at the Institute for Advanced Study, where the second author received support from the Charles Simonyi endowment.

\section{Triangulations and local face modules}

\subsection{Triangulations of simplices}\label{sec:triang}

We now recall various notions of triangulations of simplices. See \cite[Section 2]{AthanasiadisSurvey} for a detailed summary. 
Recall from the introduction that a \emph{topological triangulation} of a simplex is an order-preserving  map $\sigma \colon \Gamma \to 2^V$ of posets from an abstract simplicial complex $\Gamma$ to a simplex $2^V$, 
such that, for every subset $U \subset V$, 
the geometric realization of the
subcomplex $\Gamma_U := \sigma^{-1}(2^U)$ of $\Gamma$ is homeomorphic to a $(|U| - 1)$-dimensional ball, and $\sigma^{-1}(U)$ is the set of interior faces of the ball $\Gamma_U$.

\begin{remark}\label{rem:GleqsigmaG}
	Let $G$ be a face of $\Gamma$. Then 
$G$ is an interior face of $\Gamma_{\sigma(G)}$ which is a ball of dimension $|\sigma(G)| - 1$, and hence  $|G| \le |\sigma(G)|$. 
\end{remark}

\begin{remark}\label{rem:homology1}
A \emph{homology triangulation} over a field $k$ 
is an order-preserving  map $\sigma \colon \Gamma \to 2^V$ of posets from an abstract simplicial complex $\Gamma$ to a simplex $2^V$, 
such that, for every subset $U \subset V$, 
the
subcomplex $\Gamma_U$ of $\Gamma$ is a $(|U| - 1)$-dimensional homology ball over $k$, and $\sigma^{-1}(U)$ is the set of interior faces of the homology ball $\Gamma_U$.
The definition and properties of the local face module of quasi-geometric triangulations that are discussed in this section generalize to quasi-geometric homology triangulations over $k$. 
\end{remark}

We say that a topological triangulation $\sigma \colon \Gamma \to 2^V$ is 
\begin{enumerate}
\item\label{i:qg} \emph{quasi-geometric} if there is no face $G$ of $\Gamma$ and proper face $F \subset \sigma(G)$
such that $|F| < |G|$ and $\sigma(\{ j\}) \subset F$ for all vertices $j \in G$. 

\item\label{i:vi} \emph{vertex-induced}  
if there is no face $G$ of $\Gamma$ and proper face $F \subset \sigma(G)$ such that $\sigma(\{ j\}) \subset F$ for all vertices $j \in G$.

\item\label{i:geometric} \emph{geometric} if there is a triangulation of the geometric realization of $2^V$ into a (convex) geometric realization of $\Gamma$. 

\item\label{i:regular} \emph{regular} if it is geometric, and the triangulation can be chosen to be the projection of the lower faces of the convex hull of a polyhedron. 
\end{enumerate}
The above conditions are listed in increasing order of strength, i.e., \eqref{i:regular} $\implies$ \eqref{i:geometric} $\implies$ \eqref{i:vi} $\implies$ \eqref{i:qg}.

Consider a topological triangulation $\sigma \colon \Gamma \to 2^V$ of a $(d - 1)$-dimensional simplex. Let $h(\Gamma) = (h_0,\ldots,h_d)$ be the $h$-vector of $\Gamma$ with corresponding $h$-polynomial $h(\Gamma;t) = \sum_{s = 0}^d h_s t^s$. 
The local $h$-vector $\ell(\Gamma) = (\ell_0,\ldots,\ell_d)$ of $\sigma$ and corresponding local $h$-polynomial $\ell(\Gamma;t) = \sum_{s = 0}^d \ell_s t^s$
are defined by 
\begin{equation}\label{eq:localh}
\ell(\Gamma;t) := \sum_{U \subset V} (-1)^{|V| - |U|} h(\Gamma_U;t). 
\end{equation}
For example, when $\emptyset = U \subset V$, then $h(\Gamma_U;t) = \ell(\Gamma_U;t) = 1$. Note that  $\ell(\Gamma)$ depends on $\sigma$ and not just on $\Gamma$ (see \cite[Example~2.3e]{Stanleylocalh}).

Let $k$ be a field, and let $\{ 1, \ldots, n \}$ be the vertex set of $\Gamma$. Consider the polynomial ring $k[x_1,\ldots,x_n]$, and let $x^F = \prod_{j \in F} x_j$ for any subset $F$ of $\{ 1, \ldots, n \}$. 
Recall that the face ring $k[\Gamma] = \oplus_s k[\Gamma]^s$ 
is the quotient of $k[x_1,\ldots,x_n]$ by the ideal $(x^F : F \subset \{ 1, \ldots, n \}, \, F \notin \Gamma )$. 
We say that elements $ \theta = \{ \theta_1,\ldots,\theta_d \}$ with $\theta_i \in k[\Gamma]^1$ are a linear system of parameters (l.s.o.p.) if $A_\theta(\Gamma) := k[\Gamma]/(\theta_1, \dotsc, \theta_d)$ is finite-dimensional as a $k$-vector space. In this case, $h(\Gamma)$ agrees with the Hilbert series of $A_\theta(\Gamma)$ because 
$k[\Gamma]$
is Cohen--Macaulay, see, e.g., \cite[pg. 61]{Stanleybook}. 
We say that an element $f = \sum_{j = 1}^n a_j x_j$ in $k[\Gamma]^1$ is supported on the vertices $\{ j : a_j \neq 0 \}$ of $\Gamma$, and we write $f|_F := \sum_{j \in F} a_j x_j$ for a face $F$ of $\Gamma$. We will use the following criterion to determine if $\theta$ is an l.s.o.p. 

\begin{lemma}\label{lem:Stanleycriterion}\cite[Lemma III.2.4]{Stanleybook}
	Consider elements $ \theta = \{ \theta_1,\ldots,\theta_d \}$ with $\theta_i \in k[\Gamma]^1$. Then $\theta$ is a l.s.o.p. if and only if, for every face $F$ of $\Gamma$, the restrictions $\{ \theta_1|_F,\ldots,\theta_d|_F \}$ span a $k$-vector space of dimension $|F|$. 
\end{lemma}

Recall  that $\theta$ is special if, for each $1 \le i \le d$, $\theta_i$ is supported on vertices $j$ of $\Gamma$ with $i \in \sigma(j)$. Assume that $\theta$ is a special l.s.o.p. Recall from Definition~\ref{def:localfacemodule} that $L_\theta(\Gamma)$ is the image of $(\operatorname{int} \Gamma) = (x^G : \sigma(G) = V)$ in
$A_{\theta}(\Gamma)$. 
When the triangulation $\sigma$ is quasi-geometric, $\ell(\Gamma)$ agrees with the Hilbert function of $L_{\theta}(\Gamma)$ for any special l.s.o.p. $\theta = \{ \theta_1, \dotsc, \theta_d \}$ \cite[Theorem 4.6]{Stanleylocalh}.

The \emph{excess}  of a face $F$ of $\Gamma$ is $e(F) := |\sigma(F)| - |F|$. 
The following formula was proved in \cite[Proposition~2.2]{Stanleylocalh} and is useful for computing examples:
\begin{equation}\label{eq:excessformula}
	\ell(\Gamma;t) = \sum_{F \in \Gamma} (-1)^{d - |\sigma(F)|} t^{d - e(F)} (1 - t)^{e(F)}. 
\end{equation}

\begin{example}\label{ex:semismall}
	A topological triangulation $\sigma \colon \Gamma \to 2^V$ is \emph{semi-small} if
	$2e(F) \le |\sigma(F)|$ for all $F$ in $\Gamma$. In this case, by considering terms of lowest degree in \eqref{eq:excessformula} and using the symmetry of the local $h$-vector from Theorem~\ref{thm:localhstructure}, we conclude that 
	$\ell(\Gamma;t) = |\{ F \in \Gamma : |F| = d/2, \, \sigma(F) = V\}| \cdot t^{d/2}$.
	
	For example,  $\ell(\Gamma; t) = 0$ if $d$ is odd.
	Moreover, if $\Gamma$ is quasi-geometric, then the image of $\{ x^F : |F| = d/2, \, \sigma(F) = V \}$ is a $k$-basis for $L(\Gamma)$, for example by Lemma~\ref{lem:squarefree}. 
\end{example}

\subsection{Bilinear forms on local face modules}\label{ssec:bilinear}

Let $\sigma \colon \Gamma \to 2^V$ be a quasi-geometric triangulation of a simplex of dimension $d-1$, and let $\theta = \{ \theta_1, \dotsc, \theta_d \}$ be a special l.s.o.p. for $k[\Gamma]$.
We will make use of a symmetric 
bilinear form $B \colon L_{\theta}(\Gamma) \times L_{\theta}(\Gamma) \to k$. This bilinear form is $k[\Gamma]$-invariant, in the sense that for $x \in k[\Gamma]$ and $u, v \in L_{\theta}(\Gamma)$, we have $B(x\cdot u, v) = B(u, x \cdot v)$. The bilinear form was constructed in \cite[Corollary 4.19]{Stanleylocalh} using an identification of the canonical module of $k[\Gamma]$ with $(\operatorname{int} \Gamma)$. 
We give a different description of it. 

Let $\hat{\Gamma}$ be the triangulation of a $(d-1)$-dimensional sphere obtained from $\Gamma$ by adding a vertex $c$, and for each face $G$ of $\Gamma$ which is  contained in the boundary of $\Gamma$, i.e., with $\sigma(G) \neq V$, adding the face $G \cup \{c\}$. Set $\hat{\theta}_i = \theta_i - x_c$. 
It follows from Lemma~\ref{lem:Stanleycriterion} 
that $\hat{\theta}_1, \dotsc, \hat{\theta}_d$ is an l.s.o.p. for $k[\hat{\Gamma}]$. 
Let $A_{\hat{\theta}}(\hat{\Gamma}) = 
k[\hat{\Gamma}]/(\hat{\theta}_1, \dotsc, \hat{\theta}_d)$. 
It is known that $A_{\hat{\theta}}(\hat{\Gamma})$ is an artinian Gorenstein algebra \cite[pg. 65]{Stanleybook}: there is an isomorphism $\deg \colon A_{\hat{\theta}}^{d}(\hat{\Gamma}) \to k$ such that the pairing
$A_{\hat{\theta}}^{s}(\hat{\Gamma}) \times  A_{\hat{\theta}}^{d - s}(\hat{\Gamma}) \to k$ defined by
$(x, y) \mapsto \deg(x \cdot y)$ is nondegenerate. We choose an orientation of $\hat{\Gamma}$ and use the standard normalization of the degree map \cite{BrionStructurePolytopeAlgebra}, see \cite[Section 2.5]{KaruXiao}.
For $y \in A_{\hat{\theta}}(\hat{\Gamma})$, let $\operatorname{ann}(y)$ denote the annihilator of $y$ in $A_{\hat{\theta}}(\hat{\Gamma})$. Note that if $G$ is an interior face,  then $G \cup \{c\}$ is not a face of $\hat{\Gamma}$, so $x^G \in \operatorname{ann}(x_c)$.

\begin{lemma}\label{lem:annlocalface}
There is a graded isomorphism of $k[\Gamma]$-modules between $\operatorname{ann}(x_c)/((x_c) \cap \operatorname{ann}(x_c))$ and $L_{\theta}(\Gamma)$, given by sending the class of $x^G$ to the class of $x^G$ for every interior face $G$. 
\end{lemma}

\begin{proof}
As $A_{\hat{\theta}}(\hat{\Gamma})/(x_c)$ is naturally identified with $A_{\theta}(\Gamma) = k[\Gamma]/(\theta_1, \dotsc, \theta_d)$, it suffices to show that $\operatorname{ann}(x_c) = (\operatorname{int} \Gamma) \subset A_{\hat{\theta}}(\hat{\Gamma})$.
One inclusion is easy: 
as observed above, if $G$ is interior, then $x^G \in \operatorname{ann}(x_c)$.
We therefore have a surjective map 
$$A_{\hat{\theta}}(\hat{\Gamma})/(\operatorname{int} \Gamma) \to A_{\hat{\theta}}(\hat{\Gamma})/\operatorname{ann}(x_c).$$
Note that $A_{\hat{\theta}}(\hat{\Gamma})/\operatorname{ann}(x_c)$ is an artinian Gorenstein algebra with socle in degree $d-1$. 
Indeed, the pairing 
between graded pieces of degrees $s$ and $d - 1 - s$ defined by
$(x, y) \mapsto \deg(x \cdot y \cdot x_c)$ is nondegenerate. 
Also, 
$A_{\hat{\theta}}(\hat{\Gamma})/(\operatorname{int} \Gamma)$
is the quotient of the face ring of the cone over the boundary of $\Gamma$ by an l.s.o.p. 
Using one of the elements of the l.s.o.p., we can write the variable $x^c$ corresponding to the cone vertex in terms of the other vertices. 
This identifies $A_{\hat{\theta}}(\hat{\Gamma})/(\operatorname{int} \Gamma)$ with the quotient of $k[\operatorname{lk}_{c}(\hat{\Gamma})]$, the face ring of the link of $c$ in $\hat{\Gamma}$, by an l.s.o.p. 
Because ${\Gamma}$ is a ball of dimension $d-1$, $\operatorname{lk}_c(\hat{\Gamma})$ is a sphere of dimension $d-2$, and so this quotient of $k[\operatorname{lk}_{c}(\hat{\Gamma})]$ is an artinian Gorenstein algebra with socle in degree $d-1$. The result follows, as a surjective map between artinian Gorenstein algebras whose socles are in the same degree is an isomorphism. 
\end{proof}

\begin{lemma}
The restriction of the Poincar\'{e} pairing to $\operatorname{ann}(x_c)$ induces a nondegenerate bilinear form $B \colon L_{\theta}(\Gamma) \times L_{\theta}(\Gamma) \to k$.
\end{lemma}

\begin{proof}
For any $y \in \operatorname{ann}(x_c)$ and $z \in (x_c)$, we have $\deg(y \cdot z) = 0$, so the pairing descends to $L_{\theta}(\Gamma)$ by Lemma~\ref{lem:annlocalface}. The pairing is nondegenerate because $\operatorname{ann}(\operatorname{ann}(x_c)) = (x_c)$ as $A_{\hat{\theta}}(\hat{\Gamma})$ is an artinian Gorenstein algebra. 
\end{proof}

To summarize, given $u, v \in L_{\theta}(\Gamma)$, we can compute $B(u, v)$ by lifting $u$ and $v$ to $\operatorname{ann}(x_c) \subset A_{\hat{\theta}}(\hat{\Gamma})$, multiplying them, and then computing the degree in $A_{\hat{\theta}}(\hat{\Gamma})$.

\begin{remark}\label{rem:Bindependentfield}
	 In fact, if $u, v \in L_{\theta}(\Gamma)$ are the images of a $\Z$-linear combination of monomials, then $B(u,v)$ is the image of an element of $\Z(a_{i,j})_{1 \le j \le n, \, i \in \sigma(j)}$ in $K$. 	This follows from the construction of the degree map. See, for example, \cite{LNS}*{(1),(4)}. 
\end{remark}

\subsection{Generating the local face module}

We will need to prove several results giving generators for local face modules. See also Corollary~\ref{cor:dminus1gen}. Let $\sigma \colon \Gamma \to 2^V$ be a quasi-geometric triangulation of a simplex of dimension $d-1$, and let $\theta = \{ \theta_1, \dotsc, \theta_d \}$ be a special l.s.o.p. for $K[\Gamma]$.
The \emph{underlying face} of a nonzero monomial $x_1^{a_1} \dotsb x_n^{a_n}$ in $K[\Gamma]$ is $\{j : a_j \not= 0\}$. Because the monomial is assumed to be nonzero, this is indeed a face of $\Gamma$.

\begin{lemma}\label{lem:squarefree}
	The classes $\{x^G : G \text{ interior face}\}$ span $L_\theta(\Gamma)$. 
\end{lemma}

\begin{proof}
	By definition, $L_\theta(\Gamma)$ is spanned by monomials in $(\operatorname{int} \Gamma)$. Given such a monomial which is not squarefree, we can write it as $x_j \cdot x^J$, where the underlying face of $x^J$ is an interior face $F$ and $j \in F$. 
	By Lemma~\ref{lem:Stanleycriterion}, 
	we can take a $k$-linear combination of the $\theta_i$ so that the coefficient of $x_j$ is $1$ and the coefficient of $x_{j'}$ is $0$ for all $j' \in F \setminus \{j\}$. Using this, we can rewrite $x_j \cdot x^J$ in terms of monomials that have larger support. Because $F$ is interior, the underlying faces of these monomials will still be interior. 
\end{proof}

\begin{lemma}\label{lem:generation}
	If $\Gamma$ is vertex-induced, then $L_\theta(\Gamma)$ is spanned by monomials whose underlying face is interior and does not contain any vertices of excess $0$. 
\end{lemma}

\begin{proof}
	Let $x^J = x_{j_1}^{a_1} \dotsb x_{j_t}^{a_t}$ be a monomial whose underlying face is $\{j_1, \dotsc, j_t\}$, and assume that $\{j_1, \dotsc, j_t\}$ is interior. Because $\Gamma$ is vertex-induced, $V = \sigma(j_1) \cup \dotsb \cup \sigma(j_t)$. Suppose that $j_1$ has excess $0$, with $\sigma(j_1) = \{i_1\}$. Then $j_1$ is the unique vertex of excess $0$ such that $i_1 \in \sigma(j_1)$. 
	We may write $\theta_{i_1} = a_{i_1,j_1} x_{j_1} + \sum_{j} a_{i_1,j'} x_{j'}$, where $a_{i,j} \in k$ and  $j'$ varies over all vertices of $\Gamma$ such that $i_1 \in \sigma(j')$ and $j'$ does not have excess $0$. 
	Moreover, for any $i \neq i_1$, $i \notin \sigma(j_1)$ implies that $\theta_i|_{\{j_1\}} = 0$, and Lemma~\ref{lem:Stanleycriterion} implies that $\theta_{i_1}|_{\{j_1\}} = a_{i_1,j_1} x_{j_1} \neq 0$.
	Substituting the relation 
	$x_{j_1} = -\sum_{j'} a_{i_1,j_1}^{-1}a_{i_1,j'} x_{j'} \in
	 A_{\theta}(\Gamma)$
	into $x^J$ and expanding, we express $x^J$ as a linear combination of monomials 
	with a strictly smaller number of vertices of excess $0$ in their underlying faces than $x^J$. Moreover, 
	the underlying faces of these monomials are all interior. Indeed, this follows since $\Gamma$ is vertex-induced and
	the image of the underlying faces under $\sigma$ contains both $i_1$ and
	$V \smallsetminus \{ i_1 \} \subset \sigma(j_2) \cup \dotsb \cup \sigma(j_t)$. 
	Repeating this gives the result.
\end{proof}

\begin{remark}
	If $\Gamma$ is not vertex-induced, then the conclusion of Lemma~\ref{lem:generation} can fail, see Example~\ref{ex:nonvertexinduced}.
\end{remark}

\section{Characteristic 2 results}

Throughout this section, we consider a quasi-geometric triangulation $\sigma \colon \Gamma \to 2^V$. 
We fix $k$ to be a field of characteristic $2$ and set $K = k(a_{i,j})_{1 \le j \le n, \, i \in \sigma(j)}$. We use a generic special l.s.o.p., i.e., $\theta_i = \sum_{\sigma(j) \ni i} a_{i,j} x_j$. 

\subsection{Differential operators}

If $L$ is an $m \times n$ matrix of nonnegative integers for some positive integer $m$, then define 
\[ x^L \coloneqq \prod_{i,j} x_j^{L_{i,j}} = \prod_j x_j^{\sum_i L_{i,j}}.\] 

For each $(i, j)$ with $i \in \sigma(j)$, 
the differential operator $\frac{\partial}{\partial a_{i,j}}$ acts on $K$. 
Let $I$ be a $d \times n$ matrix of nonnegative integers. We say that $I$ is \emph{valid} if $I_{i,j} = 0$ whenever $i \not \in \sigma(j)$ and 
all row sums are $1$, i.e.,
each row of $I$ contains a unique $1$ and all other entries are $0$. If $I$ is valid, then we set
$$\partial^I \coloneqq \prod_{i,j} \left(\frac{\partial}{\partial a_{i,j}}\right)^{I_{i,j}}.$$
Here we are thinking of $x^I$ as an element of the polynomial ring $K[x_1, \dotsc, x_n]$, but we will also consider it as an element of $K[\Gamma]$ or $K[\hat{\Gamma}]$. 

\begin{remark}\label{rem:validunderlyingsupport}
	Suppose that $I$ is valid and $x^I \in K[\Gamma]$ is nonzero. Let $G$  be the underlying face of $x^I$ in $\Gamma$. Then for every $1 \le i \le d$, there exists a vertex $j$ such that $I_{i,j} = 1$, and so $i \in \sigma(j) \subset \sigma(G)$. We deduce that $G$ is interior. 
\end{remark}

We will make use of the following identity, which was conjectured by Papadakis and Petrotou \cite{PapadakisPetrotoug} and proved by Karu and Xiao \cite{KaruXiao}. See \cite[Corollary~3.1]{KLS} for a generalization to characteristic $p$. 
Given a monomial $m = x_1^{a_1} \dotsb x_n^{a_n}$, we 
set $\sqrt{m} = x_1^{a_1/2} \dotsb x_n^{a_n/2}$ if each $a_j$ is divisible by $2$, and we set $\sqrt{m} = 0$ otherwise.

\begin{proposition}\cite[Corollary 4.2]{KaruXiao}\label{prop:KXidentity}
Let $h \in K[\hat{\Gamma}]^{s}$ for some $s \le d/2$, 
let $I$ be a valid $d \times n$ matrix of nonnegative integers, and   let $J$ be a $1 \times n$  matrix of nonnegative integers  with row sum $d - 2s$.
Then
$$\partial^I \deg(h^2 \cdot x^J) =  (\deg (h\cdot \sqrt{x^I \cdot x^J}))^2.$$
\end{proposition}

In \cite{KaruXiao}, the authors work with a completely generic l.s.o.p. See \cite[Corollary 2.10]{LNS} for a discussion of how to check that \cite[Corollary 4.2]{KaruXiao} can be adapted to l.s.o.p.s which are not fully generic, like the l.s.o.p. which is used for $A_{\hat{\theta}}(\hat{\Gamma})$.  
Using the construction of the bilinear form on the local face module, we deduce the following identity for the bilinear form with a generic special l.s.o.p.

\begin{corollary}\label{cor:KXidentity}
Let $u \in L^s(\Gamma)$ for some $s \le d/2$, 
let $I$ be a valid $d \times n$ matrix of nonnegative integers, and   let $J$ be a $1 \times n$  matrix of nonnegative integers  with row sum $d - 2s$.
Then
$$\partial^I B(u, x^J \cdot u) = (B(u, \sqrt{x^I \cdot x^J}))^2.$$
\end{corollary}

Note that Remark~\ref{rem:validunderlyingsupport} implies that the image of $\sqrt{x^I \cdot x^J}$ is a well-defined element of $L(\Gamma)$ in the statement of Corollary~\ref{cor:KXidentity}.

\subsection{Lefschetz results}

We recall Hall's marriage theorem. 

\begin{theorem}\label{thm:Hall}\cite{HallRepresentativesSubsets}
	Let $G = (X,Y,E)$ be a finite bipartite graph with bipartite sets $X$ and $Y$, and edge set $E$. For any subset $S$ of $X$, assume that 
	$|S| \le |\{ y \in Y : (x,y) \in E \textrm{ for some } x \in S \}|$.
	Then there exists an $X$-perfect matching, i.e., an injection $\iota\colon X \to Y$ such that $(x,\iota(x)) \in E$ for all $x \in X$. 
\end{theorem}

\begin{proof}[Proof of Theorem~\ref{thm:WL} when $k$ has characteristic $2$]
Let $u \in L(\Gamma)^1$ be a nonzero element. By Lemma~\ref{lem:squarefree} and the nondegeneracy of $B$, there is an interior face $G$ of $\Gamma$ with $|G| = d-1$ such that $B(u, x^G)$ is nonzero. 
By Remark~\ref{rem:GleqsigmaG},
for every subset $S$ of $G$,  we have $|S| \le |\sigma(S)|$.
Because $\Gamma$ is vertex-induced, $\sigma(S) = \{ i : i \in \sigma(j) \textrm{ for some vertex } j \in S \}$.  
By Hall's marriage theorem (Theorem~\ref{thm:Hall})
with $X = G$, $Y = \{ 1,\ldots,d\}$ and $E = \{ (j,i) : i \in \sigma(j)\}$,
there is an injection $\iota \colon G \to \{1, \dotsc, d\}$ with $\iota(j) \in \sigma(j)$ for all $j \in G$.

There is a unique $\bar{i} \in \{1, \dotsc, d\}$ which is not in the image of $\iota$. Because $G$ is interior and $\Gamma$ is vertex-induced, there is some vertex $\bar{j} \in G$ with $\bar{i} \in \sigma(\bar{j})$. Let $I$ be the $d \times n$ matrix which is $1$ at $(\iota(j), j)$ for $j \in G$ and at $(\bar{i}, \bar{j})$, and is $0$ elsewhere. By construction, $I$ is valid, and $x^I = x^G \cdot x_{\bar{j}}$. 
Let $j'$ be a vertex of $G$ distinct from $\bar{j}$, which exists because $d \ge 3$, and let $x^J$ be the unique monomial so that $x^J \cdot x_{\bar{j}} \cdot x_{j'} = x^G$. By Corollary~\ref{cor:KXidentity}, we have
$$\partial^I B(\ell \cdot u, x^J \cdot u) = \partial^I \sum_{j=1}^{n} B(u , x^J \cdot x_j \cdot u  ) = \sum_{j=1}^{n} B(u, \sqrt{x^I \cdot x^J \cdot x_j})^2.$$
Since $x^I \cdot x^J = x_{j'} (x_{\bar{j}} \cdot x^J)^2$, we have $\sqrt{x^I \cdot x^J \cdot x_j}  = 0$ for $j \neq j'$,  
and $\sqrt{x^I \cdot x^J \cdot x_{j'}} = x^G$. We deduce that
$$\partial^I B(\ell \cdot u,  x^J \cdot u) = B(u, x^G)^2.$$
In particular, $\partial^I B( \ell \cdot u, x^J \cdot u)$ is nonzero, so $\ell \cdot u$ is nonzero. 
\end{proof}

\begin{corollary}\label{cor:dminus1gen}
If $\Gamma$ is vertex-induced, then $L^{d-1}(\Gamma)$ is generated by classes of the form $x_j x^G$, where $G$ is an interior face of size $d-2$. 
\end{corollary}

\begin{proof}
Using the bilinear form $B$, the map $L^{d-2}(\Gamma) \to L^{d-1}(\Gamma)$ given by multiplication by $\ell$ is dual to the map $L^1(\Gamma) \to L^2(\Gamma)$ given by multiplication by $\ell$, so it follows from Theorem~\ref{thm:WL} that multiplication by $\ell$ induces a surjection from $L^{d-2}(\Gamma)$ to $L^{d-1}(\Gamma)$. By Lemma~\ref{lem:squarefree}, $L^{d-1}(\Gamma)$ is spanned by classes of the form $\ell \cdot x^G$, where $G$ an interior face of size $d-2$. As $\ell = x_1 + \dotsb + x_n$, this implies the result. 
\end{proof}

When $k$ has characteristic $2$, we will deduce Theorem~\ref{thm:SL} from Corollary~\ref{cor:KXidentity}. To apply Corollary~\ref{cor:KXidentity}, we will need to know that $L^s(\Gamma)$ is generated by monomials of a particular form.

\begin{lemma}\label{lem:generationmiddle}
Suppose  $\Gamma$ is vertex-induced, $d$ is even, and $L^{d/2}(\Gamma)$ is generated by monomials $x^J = x_{j_1}^{a_1} \dotsb x_{j_s}^{a_s}$ such that, for every subset $S$ of $\{1, \dotsc, s\}$, we have $|\sigma(\{j_t : t \in S\})| \ge 2\sum_{t \in S} a_t$. Then $B \colon L^{d/2}(\Gamma)\times L^{d/2}(\Gamma) \to K$ is anisotropic. 
\end{lemma}

\begin{proof}
Let $u$ be a nonzero element of $L^{d/2}(\Gamma)$. Then the nondegeneracy of $B$ implies that there is some monomial $x^J = x_{j_1}^{a_1} \dotsb x_{j_s}^{a_s}$ satisfying the hypothesis of the lemma with $B(u, x^J) \not=0$. Let $T$ be the multiset consisting of $2a_1$ copies of $j_1$, $2a_2$ copies of $j_2$, and so on. Note that $|T| = d$. For any multiset $S' \subset T$, if $S$ is the underlying set associated to $S'$, then $|S'| \le 2 \sum_{t \in S} a_t \le \sigma(\{j_t : t \in S\})$. 
By Hall's marriage theorem (Theorem~\ref{thm:Hall})
with $X = T$, $Y = \{ 1,\ldots,d\}$ and $E = \{ (j,i) : i \in \sigma(j)\}$,
there is 
a bijection $\iota \colon \{1, \dotsc, d\} \to T$ such that 
$i \in \sigma(\iota(i))$ 
for all $i$.
Let $I$ be the $d \times n$ matrix which is $1$ on the entries $(i, \iota(i))$ and $0$ otherwise. By construction, $I$ is valid, and $x^I = (x^J)^2$. By Corollary~\ref{cor:KXidentity}, we have
$$\partial^I B(u, u) = B(u, x^J)^2.$$
This implies that $\partial^I B(u,u) \not=0$, so $B(u, u) \not= 0$. 
\end{proof}

\begin{proof}[Proof of Theorem~\ref{thm:SL} when $k$ has characteristic $2$]
As $L^0(\Gamma) = 0$, the statement is trivial when $s=0$. 
When $d = 2$ or $d=4$, Lemma~\ref{lem:generation} implies that the hypothesis of Lemma~\ref{lem:generationmiddle} is satisfied, proving the anisotropy of the Hodge--Riemann form in degree $d/2$. 

Consider the case when $d=3$ and $s=1$. Let $u$ be a nonzero class in $L^1(\Gamma)$. By Corollary~\ref{cor:dminus1gen}, there is a vertex $\bar{j}$ and an interior vertex $j'$ such that $B(u,  x_{\bar{j}} \cdot x_{j'} )$ is nonzero. Choose some $\bar{i} \in \sigma(\bar{j})$, and let $I$ be the $3 \times n$ matrix which is $1$ in the entries $(\bar{i}, \bar{j})$ and $(i', j')$ for each $i' \in \{1, 2, 3\} \setminus \{\bar{i}\}$, and is $0$ in all other entries. Then $I$ is valid, $x^I = x_{\bar{j}} \cdot x_{j'}^2$, and by Corollary~\ref{cor:KXidentity},
$$\partial^I B(\ell \cdot u, u) = \partial^I \sum_{j=1}^{n} B(u, x_j \cdot u) = B(u, x_{\bar{j}} \cdot x_{j'})^2.$$
Therefore $B(\ell \cdot u, u)$ is nonzero, proving that anisotropy holds. 

The remaining case is when $d=4$ and $s=1$. Let $u \in L^1(\Gamma)$ be a nonzero class. Then $\ell \cdot u$ is nonzero by Theorem~\ref{thm:WL}, so the anisotropy of $B$ on $L^2(\Gamma)$ implies that $B(\ell \cdot u, \ell \cdot u) = B(\ell^2 \cdot u, u)$ is nonzero. 
\end{proof}

\begin{remark}\label{rem:homologylefschetz}
The results in this section generalize to vertex-induced homology triangulations over a field of characteristic $2$. If stronger results were known about spanning sets for local face modules, then one could prove Lefschetz results for local face modules of characteristic $p$ vertex-induced homology triangulations for any prime $p$ using the results in \cite{KLS}. For example, if $d \ge p+1$ and $L^{d-1}(\Gamma)$ is generated by monomials in the interior vertices over a field of characteristic $p$, then the proof of Theorem~\ref{thm:WL} shows that multiplication by $\ell$ induces an injection from $L^1(\Gamma)$ to $L^2(\Gamma)$. 
\end{remark}

\section{Characteristic 0 results}\label{sec:char0}

In this section, we prove Lefschetz results for local face modules in characteristic $0$. We first complete the proof of Theorems~\ref{thm:WL} and \ref{thm:SL} by deducing the case when $k$ has characteristic $0$ from the case when $k$ has characteristic $2$. This argument is fairly standard, see, e.g., \cite[Section 5]{KaruXiao} or \cite[Proposition 3.5]{KLS}. 
We then relate local face modules to the cohomology of toric varieties and use that to prove Theorem~\ref{thm:regular}.

\begin{proof}[Proof of Theorem~\ref{thm:WL} when $k$ has characteristic $0$]
Let $M$ be the matrix whose rows are labeled by interior vertices $j$ of $\Gamma$ and whose columns are labeled by interior faces $G$ of size $d-2$, where the corresponding entry is $B(\ell \cdot x_j, x^G)$, using the bilinear form on the generic local face module over a field $k$ of characteristic $0$. Define $M_2$ similarly, except using the bilinear form on the generic local face module over a field of characteristic $2$. 
By Remark~\ref{rem:Bindependentfield}, the entries of $M_2$ and $M$ are the images of elements in $\Z(a_{i,j})$. In particular, 
the entries of $M_2$ lie in $\mathbb{F}_2(a_{i,j})$ and are independent of the choice of field of characteristic $2$, and similarly for $M$. By the characteristic $2$ case of Theorem~\ref{thm:WL}, the rank of $M_2$ is equal to $\dim L^1(\Gamma) = \ell_1(\Gamma)$. We obtain $M_2$ by reducing $M$ modulo $2$, and so the rank of $M$ is equal to $\ell_1(\Gamma)$, which implies that multiplication by $\ell$ is injective. 
\end{proof}

\begin{proof}[Proof of Theorem~\ref{thm:SL} when $k$ has characteristic $0$]
Recall that the $2$-adic valuation on $\mathbb{Q}$ is $\nu \colon \mathbb{Q} \to \Z \cup \{ \infty \}$, where $\nu(0) = \infty$, $2^{\nu(m)}$ is the largest power of $2$ dividing $m$ for a nonzero integer $m$, and $\nu(\frac{m}{m'}) = \nu(m) - \nu(m')$ for nonzero integers $m,m'$. By Chevalley's Extension Theorem 	\cite{EPValuedFields}*{Theorem~3.1.2}, there exists an extension of the  $2$-adic valuation to a valuation $\nu \colon k \to  \Gamma \cup \{ \infty \}$, for some ordered abelian group $\Gamma$. Consider the valuation ring $R = \{ x \in k : \nu(x) \ge 0 \}$ with maximal ideal $m = \{ x \in k : \nu(x) > 0 \}$ and  residue field $\kappa = R/m$. 
Let $L(\Gamma)_{\kappa}$ denote the generic local face module over $\kappa(a_{i,j})$, and let $L(\Gamma)_k$ denote the generic local face module over $K$. For each $s \le d/2$, choose a basis for $L^s(\Gamma)_{\kappa}$ consisting of monomials. The same monomials form a basis for $L^s(\Gamma)_k$ as well. Given a nonzero element $u$ of $L^s(\Gamma)_k$, we can find an element $\lambda \in k[a_{i,j}]$ so that, when we express $\lambda u$ in this basis, the coefficients are polynomials in $R[a_{i,j}]$, and at least one of the coefficients is nonzero in $\kappa[a_{i,j}]$.
Let $\overline{\lambda u} \in L^s(\Gamma)_{\kappa}$ be the result of reducing all coefficients to $\kappa$, which is nonzero. Then $B(\overline{\lambda u}, \ell^{d - 2s} \cdot \overline{\lambda u}) \not=0$. But this is obtained by reducing $B(\lambda u, \ell^{d - 2s} \cdot \lambda u) = \lambda^2 B(u, \ell^{d - 2s} \cdot u)$ to $\kappa$, so $B(u, \ell^{d - 2s} \cdot u) \not= 0$. 
\end{proof}

\medskip

In order to prove Theorem~\ref{thm:regular}, we first discuss some aspects of the cohomology of algebraic varieties. See \cite{dCM} for a reference. 
Let $Y$ be a complex projective variety, and let $f \colon X \to Y$ be a projective morphism with $X$ a rationally smooth variety, e.g., a simplicial toric variety. Let $d = \dim X$. There is an increasing filtration on the singular cohomology of $X$, $0 \subseteq P_0 \subseteq P_1 \subseteq \dotsb \subseteq P_{2d} = H^*(X; \mathbb{Q})$, called the \emph{perverse filtration}, which is characterized as follows. The decomposition theorem \cite{BBDG} states that there is an isomorphism in $D^b_c(Y)$:
\begin{equation}\label{eq:decomposition}
Rf_* \underline{\mathbb{Q}}_X[d] \stackrel{\sim}{\to} \bigoplus_{\alpha} IC(Y_{\alpha}, \mathcal{L}_{\alpha})[s_{\alpha}].
\end{equation}
Here $\{Y_{\alpha}\}$ is a set of locally closed smooth subvarieties of $Y$, $\mathcal{L}_{\alpha}$ is a simple local system on $Y_{\alpha}$, $IC(Y_{\alpha}, \mathcal{L}_{\alpha})$ is the intersection cohomology perverse sheaf, and $s_{\alpha}$ is an integer. Taking hypercohomology, this induces a decomposition
$$H^*(X; \mathbb{Q}) = \bigoplus_{\alpha} \mathcal{H}(IC(Y_{\alpha}, \mathcal{L}_{\alpha})[s_{\alpha} - d]),$$
where the grading on $H^*(X; \mathbb{Q})$ goes from $0$ to $2d$. 
We set $P_b$ to be the subspace of $H^*(X; \mathbb{Q})$ spanned by $\mathcal{H}(IC(Y_{\alpha}, \mathcal{L}_{\alpha})[s_{\alpha} - d])$ for those $\alpha$ with $s_{\alpha} \ge d - b$.  While the isomorphism in \eqref{eq:decomposition} is not unique, the induced filtration is independent of this choice. 

A class $x \in H^i(X; \mathbb{Q})$, thought of as an element of $\operatorname{Hom}_{D^b_c(X)}(\underline{\mathbb{Q}}_X, \underline{\mathbb{Q}}_X[i])$, pushes forward to give a map $Rf_* \underline{\mathbb{Q}}_X \to Rf_* \underline{\mathbb{Q}}_X[i]$, and so it induces a map
$$\bigoplus_{\alpha} IC(Y_{\alpha}, \mathcal{L}_{\alpha})[s_{\alpha} - d] \to \bigoplus_{\alpha} IC(Y_{\alpha}, \mathcal{L}_{\alpha})[s_{\alpha} + i - d].$$
The definition of the perverse $t$-structure implies that there are no nonzero maps from $IC(Y_{\alpha}, \mathcal{L}_{\alpha})[s_{\alpha} - d]$ to $IC(Y_{\beta}, \mathcal{L}_{\beta})[s_{\beta} + i - d]$ unless $s_{\beta} - s_{\alpha} + i \ge 0$. Therefore, for each $b$, we have $x \cdot P_b \subseteq P_{b + i}$. In particular, the associated graded $\operatorname{Gr}^{\bullet} P$ is an $H^*(X; \mathbb{Q})$-module, where $H^i(X; \mathbb{Q}) \cdot \operatorname{Gr}^{\bullet} P \subseteq \operatorname{Gr}^{\bullet + i} P$. Note that $\operatorname{Gr}^{\bullet} P$ is bigraded, where the grading arising from the grading on $H^*(X; \mathbb{Q})$ is not indicated in the notation. 

The intersection cohomology perverse sheaves associated to simple local systems are simple objects in the category of perverse sheaves. In particular, there are no nonzero maps $IC(Y_{\alpha}, \mathcal{L}_{\alpha}) \to IC(Y_{\beta}, \mathcal{L}_{\beta})$ unless $IC(Y_{\alpha}, \mathcal{L}_{\alpha})$ is isomorphic to  $IC(Y_{\beta}, \mathcal{L}_{\beta})$. Collecting the summands in \eqref{eq:decomposition} which are isomorphic to a shift of some fixed $IC(Y_{\alpha}, \mathcal{L}_{\alpha})$ and taking hypercohomology, we obtain an $H^*(X; \mathbb{Q})$-module summand of $\operatorname{Gr}^{\bullet} P$.

\medskip

Let $\sigma \colon \Gamma \to 2^V$ be a regular triangulation of a simplex. Identify $2^V$ with the standard simplex in $\mathbb{R}^d$.
Because $\Gamma$ is regular, it is possible to identify the vertices of $\Gamma$ with points of the standard simplex and choose a height $h(j)$ for each vertex $j$ so that the induced regular subdivision of the standard simplex is $\Gamma$. 
Because $\Gamma$ is a triangulation, we can perturb the coordinates of the vertices of $\Gamma$ so that they lie in $\mathbb{Q}^d$ without changing the combinatorial structure. We may also assume that the heights are rational. 

Let $\Sigma_{\Gamma}$ be the fan over $\Gamma$, i.e., $\Sigma_{\Gamma}$ is a fan whose support is $\mathbb{R}^d_{\ge 0}$, and whose cones are the cones generated by the faces of $\Gamma$. Let $X_{\Gamma}$ be the associated toric variety. Because $\Sigma_{\Gamma}$ is a simplicial fan, $X_{\Gamma}$ is rationally smooth. 

Because the support of $\Sigma_{\Gamma}$ is $\mathbb{R}^d_{\ge 0}$, there is a proper 
map $f \colon X_{\Gamma} \to \mathbb{A}^d$. We can then apply the decomposition theorem to $Rf_* \underline{\mathbb{Q}}_{X_{\Gamma}}[d]$. By \cite[Theorem 5.1]{deCataldoMiglioriniMustata18}, the strata appearing are torus-orbits in $\mathbb{A}^d$, and all local systems appearing are trivial. For $S \in 2^V$, let $V(S)$ denote the coordinate subspace of $\mathbb{A}^d$ where the coordinates labeled by $S$ are all $0$. 
The decomposition theorem then takes the form
$$Rf_* \underline{\mathbb{Q}}_{X_{\Gamma}}[d] \stackrel{\sim}{\to} \bigoplus_{S \in 2^V} \bigoplus_{s \in \mathbb{Z}} \underline{\mathbb{Q}}_{V(S)}^{\oplus \ell(S, s)}[d - |S| + s],$$
where $\ell(S, s) = 0$ if $s + |S|$ is odd. 
In \cite[Section 5]{Stanleylocalh}, Stanley shows that $\ell(V, d-2k) = \ell_{k}(\Gamma)$ by showing that it satisfies the same recursion as the local $h$-polynomial, i.e., the local $h$-polynomial enumerates the number of summands which are supported at the origin.

Because $\Gamma$ is a regular subdivision, $f$ is a projective morphism. Explicitly, each vertex $j$ of $\Gamma$ corresponds to a torus-invariant divisor $D_j$. Let $\lambda_j$ be the sum of the coordinates of the primitive lattice point on the ray spanned by $j$. The class $\sum_j h(j) \lambda_j [D_j] \in H^2(X_{\Gamma}; \mathbb{Q})$ is relatively ample. 
The relative Hard Lefschetz theorem implies that the sequence $(\ell(V, d - 2k))_{k=0, 1, \dotsc, d}$ is unimodal, thereby proving the unimodality of the local $h$-polynomial \cite[Theorem~5.2]{Stanleylocalh}. We will prove a stronger statement below.

\medskip

Let $\Sigma_{\hat{\Gamma}}$ be the fan obtained by adding to $\Sigma_{\Gamma}$ the ray spanned by $(-1, \dotsc, -1)$, and for each face $F$ of $\Gamma$ which is not interior, adding the cone generated by $(-1, \dotsc, -1)$ and the vertices of $F$. Note that $\Sigma_{\hat{\Gamma}}$ is a complete fan. 

Let $X_{\hat{\Gamma}}$ be the toric variety corresponding to $\Sigma_{\hat{\Gamma}}$. Let $\hat{f} \colon X_{\hat{\Gamma}} \to \mathbb{P}^d$ be the map induced by the fact that $\Sigma_{\hat{\Gamma}}$ is a subdivision of the fan of $\mathbb{P}^d$. The map $\hat{f}$ is projective: if we add a sufficiently large multiple of the divisor corresponding to the ray spanned by $(-1, \dotsc, -1)$ to $\sum_j h(j) \lambda_j [D_j] \in H^2(X_{\hat{\Gamma}}; \mathbb{Q})$, then the associated piecewise-linear function is convex, and so this class is ample.

\begin{proof}[Proof of Theorem~\ref{thm:regular}]
We will apply the decomposition theorem to $R\hat{f}_* \underline{\mathbb{Q}}_{X_{\hat{\Gamma}}}[d]$. By \cite[Theorem 5.1]{deCataldoMiglioriniMustata18}, the summands which appears are shifts of the constant sheaf on torus-orbit closures in $\mathbb{P}^d$. Because $\hat{f}$ is an isomorphism over the generic point of any torus-orbit closure which is disjoint from $\mathbb{A}^d \subset \mathbb{P}^d$, all summands which appear are shifts of constant sheaves on a torus-orbit closure which intersects $\mathbb{A}^d$. We deduce that
\begin{equation}\label{eq:hatfdecomp}
R\hat{f}_* \underline{\mathbb{Q}}_{X_{\hat{\Gamma}}}[d] \stackrel{\sim}{\to} \bigoplus_{S \in 2^V} \bigoplus_{s \in \mathbb{Z}} \underline{\mathbb{Q}}_{\mathbb{P}^{V \setminus S}}^{\oplus \ell(S, s)}[d - |S| + s],
\end{equation}
where $\mathbb{P}^{V \setminus S}$ is the closure of the coordinate subspace $V(S)$ of $\mathbb{A}^d$ and $\ell(S, s) = 0$ if $s + |S|$ is odd. After halving the degrees, we can identify $H^*(X_{\hat{\Gamma}}; \mathbb{Q})$ with the ring denoted $A_{\hat{\theta}}(\hat{\Gamma})$ in Section~\ref{ssec:bilinear}, where $k=\mathbb{Q}$ and the l.s.o.p. is the one induced by the globally linear functions on $\mathbb{R}^d$ \cite[Section 5.2]{Fulton}. Explicitly, if $\rho_j$ is the primitive lattice point on the ray spanned by $j$ and $e_i$ is the $i$th standard basis vector, then $\theta_i = \sum_{j \in \Gamma} \langle e_i, \rho_j \rangle x_j$. Let $x_c \in A_{\hat{\theta}}^1(\hat{\Gamma}) = H^2(X_{\hat{\Gamma}}; \mathbb{Q})$ be the class of the divisor corresponding to the ray spanned by $(-1, \dotsc, -1)$. Note that $x_c$ is the pullback of the hyperplane class on $\mathbb{P}^d$. For $k=0, \dotsc, d-1$, let $L_k = \mathbb{Q}[x_c]/(x_c^{k+1})$,  thought of as a module over $\mathbb{Q}[x_c]/(x_c^{d+1})$. Here $x_c$ is in degree $1$. 
Taking the hypercohomology of \eqref{eq:hatfdecomp}, we have an isomorphism
$$A_{\hat{\theta}}(\hat{\Gamma}) \stackrel{\sim}{\to} \bigoplus_{S \in 2^V} \bigoplus_{s \in \mathbb{Z}} L_{d - |S|}^{\oplus \ell(S, s)}[(- |S| + s)/2].$$
This decomposition gives the Jordan block decomposition of the action of $x_c$ on $A_{\hat{\theta}}(\hat{\Gamma})$. Let $o$ be the origin, and note that the hypercohomology of the summands isomorphic to a shift of $\underline{\mathbb{Q}}_{o}$ is $\oplus_{s \in \mathbb{Z}} L_0^{\oplus \ell(V, s)}[(- d + s)/2]$. The intersection of $P_b$ with this subspace of $H^*(X_{\hat{\Gamma}}; \mathbb{Q})$ is $\oplus_{s \ge d - b} L_0^{\oplus \ell(V, s)}[(- d + s)/2]$.
By Lemma~\ref{lem:annlocalface}, $L_{\theta}(\Gamma)$ is identified with $\operatorname{ann}(x_c)/((x_c) \cap \operatorname{ann}(x_c))$, and we see that $\operatorname{ann}(x_c)/((x_c) \cap \operatorname{ann}(x_c))$ is identified with $\oplus_{s \in \mathbb{Z}} L_0^{\oplus \ell(V, s)}[(- d + s)/2]$. This identifies $L_{\theta}(\Gamma)$ with a summand of the associated graded of $H^*(X_{\hat{\Gamma}}; \mathbb{Q})$ with respect to the perverse filtration, and so the relative Hard Lefschetz theorem \cite[Theorem 1.6.3]{dCM} implies that the action of any ample class induces an isomorphism from $L_{\theta}^{t}(\Gamma)$ to $L_{\theta}^{d-t}(\Gamma)$ for each $t$. 

It remains to deduce that $L(\Gamma)$ has the strong Lefschetz property from this. Given an element $y \in \mathbb{Q}[\Gamma]^1$, multiplication by $y^{d - 2t}$ induces an isomorphism from $L^t_{\theta}(\Gamma)$ to $L^{d - t}_{\theta}(\Gamma)$ if and only if the bilinear form $L^t_{\theta}(\Gamma) \times L^t_{\theta}(\Gamma) \to \mathbb{Q}$ given by $(u, v) \mapsto B(u, y^{d-2t} \cdot v)$ is nondegenerate. Let $M_{\theta}$ be the matrix whose rows and columns are indexed by interior faces of size $t$ and whose entry labeled by faces $F$ and $G$ is $B(x^F, y^{d - 2t} \cdot x^G)$, using the bilinear form on $L_{\theta}(\Gamma)$. Then multiplication by $y^{d - 2t}$ is an isomorphism if and only if the rank of $M_{\theta}$ is equal to $\ell_t$. Let $M$ be the matrix which is defined analogously except using the bilinear form on $L(\Gamma)$. The entries of $M$ are rational functions, and we obtain $M_{\theta}$ by evaluating these rational functions. Therefore the rank of $M$ is at least the rank of $M_{\theta}$, and so if multiplication by $y^{d - 2t}$ induces an isomorphism from $L^t_{\theta}(\Gamma)$ to $L^{d - t}_{\theta}(\Gamma)$, then multiplication by $y^{d - 2t}$ also induces an isomorphism from $L^t(\Gamma)$ to $L^{d - t}(\Gamma)$. 
\end{proof}

\section{Discussion and examples}

\subsection{Examples}

We first recall some properties of the join of two subdivisions. 
Let $\sigma\colon \Gamma \to 2^V$ and $\sigma'\colon \Gamma' \to 2^{V'}$ be topological subdivisions of simplices with $|V| = d$ and $|V'| = d'$. Let $\Gamma * \Gamma' = \{ (F,F') : F \in \Gamma, F' \in \Gamma' \}$ be the join of $\Gamma$ and $\Gamma'$. Then there is a corresponding topological subdivision 
$\sigma \times \sigma' \colon \Gamma * \Gamma' \to 2^{V \sqcup V'}$ defined by $\sigma(F,F') = \sigma(F) \cup \sigma(F')$. Moreover, if $\sigma$ and $\sigma'$ are regular subdivisions, then $\sigma * \sigma'$ is a regular subdivision. Let $k$ be a field.  Then the corresponding face ring is $k[\Gamma * \Gamma'] \cong k[\Gamma] \otimes_k k[\Gamma']$. If 
$\theta = \{ \theta_1, \dotsc, \theta_d \}$ and $\theta' = \{ \theta_1', \dotsc, \theta_{d'}' \}$ are special l.s.o.p.s for  $k[\Gamma]$ and $k[\Gamma']$ respectively, then $(\theta, \theta') := \{ \theta_1 \otimes 1, \dotsc, \theta_d \otimes 1,  1 \otimes \theta_1', \dotsc, 1 \otimes \theta_{d'}' \}$ is a special l.s.o.p. for $k[\Gamma * \Gamma']$, and,   moreover,  any special l.s.o.p. for $k[\Gamma * \Gamma']$ has this form. Then $A_{(\theta, \theta')}(\Gamma * \Gamma') \cong  A_\theta(\Gamma) \otimes_k A_{\theta'}(\Gamma')$ and  $L_{(\theta, \theta')}(\Gamma * \Gamma') \cong  L_\theta(\Gamma) \otimes_k L_{\theta'}(\Gamma')$.

\begin{example}\label{ex:counterexamples}
Let $V = \{ v_1, v_2 , v_3 \}$ and let $\sigma \colon \Gamma \to 2^V$ be 
the regular subdivision of a standard simplex with a unique interior point $w$, i.e., $\Gamma$ has vertices $v_1,v_2,v_3, w$ with $\sigma(w) = V$, $\sigma(v_s) = v_s$ for $s \in \{ 1,2,3\}$, and minimal nonface $\{ v_1, v_2, v_3 \}$. 
Fix a field $k$ of characteristic $p$ (here we allow $p = 0$), 
and let $\theta$ be a special l.s.o.p. for $k[\Gamma]$. 
Then $A_\theta(\Gamma) \cong k[x]/(x^3)$ and $L_\theta(\Gamma) = k x \oplus k x^2$, where $x = x_w$.

For $t$ a positive integer, consider the regular subdivision obtained by taking the join $\sigma \times \cdots \times \sigma$ of $\sigma$ with itself $t$ times. Let $\Gamma_t = \Gamma * \cdots * \Gamma$ and let $\theta_t$ be a special l.s.o.p. for $k[\Gamma_t]$. 
For $1 \le s \le t$,  let $x_s$ be the variable in $k[\Gamma_t]$ corresponding to the interior vertex in the $s$th copy of $\Gamma$ in $\Gamma_t$.
By the discussion above,
$A_{\theta_t}(\Gamma_t) \cong k[x_1,\ldots,x_t]/(x_1^3,\ldots,x_t^3)$, and 
$L_{\theta_t}(\Gamma_t)$ has a $k$-basis $\{ x_1^{l_1}\cdots x_t^{l_t} : l_s \in \{ 1,2 \} \textrm{ for } 1 \le s \le t \}$. 
The corresponding nondegenerate bilinear form $B \colon L_{\theta_t}(\Gamma_t) \times L_{\theta_t}(\Gamma_t) \to k$ satisfies the property that
$B(x_1^{l_1}\cdots x_t^{l_t}, x_1^{l_1'}\cdots x_t^{l_t'})$ is nonzero if and only if $l_s + l_{s}' = 3$ for $1 \le s \le t$.

Consider a nonzero element $y = \sum_{s = 1}^t a_s x_s \in A_{\theta_t}^1(\Gamma_t)$ for some $a_s \in k$. 
For $t \le m \le 2t$, let $z_m := \sum_{(l_1,\ldots,l_t)} a_1^{l_1 - 1}\cdots a_t^{l_t - 1} x_1^{l_1}\cdots x_t^{l_t} \in L^m_{\theta_t}(\Gamma_t)$, where the sum varies over all $(l_1,\ldots,l_t)$ with $l_s \in \{ 1,2 \}$ for $1 \le s \le t$, and $\sum_s l_s = m$. For example, $z_t = x_1 \cdots x_t$ and $z_{2t} = (\prod_s a_s) x_1^2 \cdots x_t^2$. We claim that $y \cdot z_m = (m + 1 - t) z_{m + 1}$ in $L_{\theta_t}(\Gamma_t)$. Indeed, we compute 
\begin{align*}
	y \cdot z_m &= \left( \sum_{s = 1}^t a_s x_s \right) \sum_{(l_1,\ldots,l_t)} a_1^{l_1 - 1}\cdots a_t^{l_t - 1} x_1^{l_1}\cdots x_t^{l_t} \\
&= \sum_{(l_1',\ldots,l_t')} |\{ s : l_s' = 2 \}|	 a_1^{l_1' - 1}\cdots a_t^{l_t' - 1} x_1^{l_1'}\cdots x_t^{l_t'}, 
\end{align*}
where  the latter sum varies over all $(l_1',\ldots,l_t')$ with $l_s' \in \{ 1,2 \}$ for $1 \le s \le t$ and $\sum_s l_s' = m + 1$. The claim follows since $|\{ s : l_s' = 2 \}| = m + 1 - t$. 

In particular, 
multiplication by $y^t$ induces a map from $L^t_{\theta_t}(\Gamma_t) = k x_1\cdots x_t$ to  $L^{2t}_{\theta_t}(\Gamma_t) = k x_1^2\cdots x_t^2$ given by multiplication by $t! \prod_s a_s$. We conclude that the strong Lefschetz property fails for  $L_{\theta_t}(\Gamma_t)$ if $0 < p \le t$. 

Similarly, multiplication by $y^{\lceil \frac{t}{2} \rceil}$ induces a map from $L^t_{\theta_t}(\Gamma_t) = k x_1\cdots x_t$ to  $L^{\lceil \frac{3t}{2} \rceil}_{\theta_t}(\Gamma_t)$, and 
$y^{\lceil \frac{t}{2} \rceil} \cdot x_1\cdots x_t = y^{\lceil \frac{t}{2} \rceil} \cdot z_t = \lceil \frac{t}{2} \rceil! z_{\lceil \frac{3t}{2} \rceil}$. Therefore if $0 < p \le \lceil \frac{t}{2} \rceil$, or, equivalently, $p > 0$ and $2p - 1 \le t$, then the above map is not injective. Because the local $h$-vector of $\Gamma_{t}$ is unimodal, this implies that the weak Lefschetz property fails for  $L_{\theta_t}(\Gamma_t)$: at least one of the maps $L_{\theta_t}^{s}(\Gamma_t) \to L_{\theta_t}^{s+1}(\Gamma_t)$ with $t \le s < \lceil \frac{3t}{2} \rceil$ is not injective. For example, when $p = 2$ and $t = 3$, the map $L^4_{\theta_t}(\Gamma_t) \to L^5_{\theta_t}(\Gamma_t)$ induced by multiplication by a nonzero element $y \in A_{\theta_3}^1(\Gamma_3)$ is not injective. In this case, $z_4$ is nonzero, but $y \cdot z_4 = 0$. 

Finally, consider the bilinear form  $L_{\theta_t}^m(\Gamma_t) \times L_{\theta_t}^m(\Gamma_t)  \to k$ given by $(u, v) \mapsto B(u, y^{3t-2m} \cdot v)$ for some $t \le m \le \lfloor \frac{3t}{2} \rfloor$. We claim that this form is not anisotropic for any $m \neq t$ and for $k$ a field of any characteristic. Indeed, consider a basis element $u = x_1^{l_1}\cdots x_t^{l_t}$ of $L_{\theta_t}^m(\Gamma_t)$ with  $l_s \in \{ 1,2 \}$ for  $1 \le s \le t$ and $\sum_s l_s = m > t$. Then there exists an index $s'$ with $l_{s'} = 2$, and it follows that $B(u, y^{3t-2m} \cdot u) = 0$. 
\end{example}

Our next goal is to recall analogues of the $h$-vector and local $h$-vector in Ehrhart theory, and present an analogue of Example~\ref{ex:counterexamples}. We refer the reader to  \cite{BarvinokIntegerPolyhedra} and \cite{BeckRobinsComputing}  for surveys on Ehrhart theory. 
Let $N$ be a lattice, and write $N_\R := N \otimes \R$. 
Let $P \subset N_\R$ be a full-dimensional lattice polytope with respect to $N$. The $h^*$-polynomial $h^*(P;t)$ and local $h^*$-polynomial $\ell^*(P;t)$  are polynomials of degree at most $\dim P$ with nonnegative integer coefficients. The $h^*$-polynomial
encodes the number of lattice points in all integer dilates of $P$, while the local $h^*$-polynomial satisfies 
the symmetry 
$\ell^*(P;t) = t^{\dim P + 1}\ell^*(P;t^{-1})$. 
For example, if $h^*(P;t) = \sum_{i = 0}^{\dim P} h^*_i t^i$ and $\ell^*(P;t) = \sum_{i = 0}^{\dim P} \ell^*_i t^i$, then $h_0^* = 1$, $\ell_0^* = 0$,  $h_1^* = |P \cap N| - \dim P - 1$, and $\ell_1^* = \ell_{\dim P}^* =  h_{\dim P}^* =  |\Int(P) \cap N|$, where $\Int(P)$ denotes the relative interior of $P$.
When $P$ is the empty set,  $h^*(P;t) = \ell^*(P;t) = 1$.

Assume further that $P$ is a simplex with vertex set indexed by $V = \{ 1, \ldots, d\}$. In this case, the $h^*$-polynomial and local $h^*$-polynomial can be inductively defined in terms of each other via the equations
\[
h^*(P;t) = \sum_{U \subset V} \ell^*(F_U;t),\: \: \ell^*(P;t) = \sum_{U \subset V}  (-1)^{|V| - |U|} h^*(F_U;t),
\]
where $F_U$ denotes the face of $P$ with vertices $U \subset V$.  
Suppose 
there exists a unimodular lattice triangulation $\Gamma$ of $P$, i.e., $\Gamma$ is a geometric triangulation of $P$ with vertices in $N$ such that for every facet $G$ of $\Gamma$ with vertices $u_1,\ldots,u_{d}$, the vectors $u_1 - u_{d}, \ldots,u_{d - 1} - u_{d}$ form a basis of $N$.  In this case, $h^*(P;t)$ and $\ell^*(P;t)$ agree with the $h$-polynomial $h(\Gamma;t)$ and local $h$-polynomial $\ell(\Gamma;t)$ respectively.

Let $\tilde{N} = N \oplus \Z$, and let 
$C_P \subset \tilde{N}_\R$ denote the cone generated by $P \times \{1\}$. 
Fix a field $k$. The lattice points $C_P \cap \tilde{N}$ naturally form a semigroup. Let $k[P]$ denote the corresponding semigroup ring. That is, $k[P]$ has a $k$-basis $\{ x^u : u \in C_P \cap \tilde{N}\}$, and $x^{u} x^{u'} = x^{u + u'}$ for any $u,u' \in C_P \cap \tilde{N}$. Projection onto the last coordinate $\hht \colon \tilde{N}_\R \times \R \to \R$  induces a natural grading on $k[P] = \oplus_i k[P]^i$ such that $x^u$ has degree  $\hht(u) \in \Z_{\ge 0}$ for all $u \in C_P \cap \tilde{N}$. 
Let $k[P, \partial P]$ be the graded ideal of $k[P]$ with $k$-basis $\{ x_u : u \in \Int(C_P) \cap \tilde{N}\}$. 
We say that an element $f = \sum_{u \in P \cap N} a_u x^{(u,1)} \in k[P]^1$ is \emph{supported} on $\{ u \in P \cap N : a_u \neq 0 \}$. 
The $k$-algebra $k[P]$ is Cohen-Macaulay with Krull dimension $d$. A collection of 
elements $ \theta = \{ \theta_1,\ldots,\theta_d \}$ with $\theta_i \in k[P]^1$ are an l.s.o.p. if $\mathscr{A}_\theta(P) := k[P]/(\theta_1, \dotsc, \theta_d)$ is finite-dimensional as a $k$-vector space. In this case, the Hilbert polynomial of $\mathscr{A}_\theta(P)$ is equal to $h^*(P;t)$. 
Let
$\sigma \colon  P \cap N \to  2^V$
be the map so that $F_{\sigma(u)}$ is the smallest face of $P$ containing $u$. Then an l.s.o.p. $\theta$ is \emph{special} if  $\theta_i$ is supported on $u \in P \cap N$ with $i \in \sigma(u)$. Special l.s.o.p.s always exist; for example, $\{ x^{(u,1)} : u \textrm{ vertex of } P \}$ is a special l.s.o.p. Given a special l.s.o.p., we define  $\mathscr{L}_{\theta}(P)$ to be the image of $k[P, \partial P]$ in $\mathscr{A}_\theta(P)$. In this case, the Hilbert polynomial of $\mathscr{L}_\theta(P)$ is equal to $\ell^*(P;t)$.  

Let $P'$ be a full-dimensional lattice simplex with respect to a lattice $N'$ of rank $d' - 1$. Let $0_N$ and $0_{N'}$ denote the origins in $N$ and $N'$ respectively. 
The \emph{free join} $P*P'$ is the 
convex hull in $N_\R \oplus N_\R' \oplus \R^2$ of $\{ (u,0_{N'},1,0) : u \in P \}$ and $\{ (0_{N},u',0,1) : u' \in P' \}$. Then $P*P'$ is a 
lattice simplex  with respect to $N \oplus N' \oplus \Z^2$ with $d + d'$ vertices. We have an isomorphism of graded $k$-algebras $\psi \colon  k[P] \otimes_k k[P'] \to k[P * P']$, defined by $\psi(x^{(u,m)} \otimes x^{(u',m')}) = x^{(u,u',m,m',m + m')}$. 
	
	If $\theta = \{ \theta_1,\ldots,\theta_d \}$ and $\theta' = \{ \theta_1',\ldots,\theta_{d'}'\}$ are special l.s.o.p.s for $k[P]$ and $k[P']$ respectively, then $(\theta,\theta') = \{ \psi(\theta_1 \otimes 1), \ldots,  \psi(\theta_d \otimes 1), \psi(1 \otimes \theta_1'), \ldots, \psi(1 \otimes \theta_{d'}') \}$ is a special l.s.o.p. for  $k[P * P']$, and,   moreover,  any special l.s.o.p. for $k[P * P']$ has this form. Then $\psi$ induces isomorphisms 
	$\mathscr{A}_\theta(P) \otimes_k \mathscr{A}_{\theta'}(P') \cong 
	\mathscr{A}_{(\theta,\theta')}(P*P')$ and 
	$\mathscr{L}_\theta(P) \otimes_k \mathscr{L}_{\theta'}(P') \cong 
\mathscr{L}_{(\theta,\theta')}(P*P')$. Finally, we have the following analogue of Example~\ref{ex:counterexamples}. 

\begin{example}\label{ex:ehrhart}
Let $P \subset \R^2$ be the convex hull of $(1,0)$, $(0,1)$ and $(-1,-1)$. Then $P$ has $4$ lattice points; its vertices and the interior lattice point $w = (0,0)$. 
Fix a field $k$ of characteristic $p$ and let $\theta$ be a special l.s.o.p. for $k[P]$. 
Then $\mathscr{A}_\theta(P) \cong k[x]/(x^3)$ and $\mathscr{L}_\theta(\Gamma) = k x \oplus k x^2$, where $x = x_w$. 
For $t$ a positive integer, let $P_t = P * \cdots * P$ be the free join of $P$ with itself $t$ times and let $\theta_t$ be a special l.s.o.p. for $P_t$. 
For $1 \le s \le t$,  let $x_s$ be the variable in $k[P_t]$ corresponding to the interior lattice point in the $s$th copy of $P$ in $P_t$. 
By the discussion above,
$\mathscr{A}_{\theta_t}(P_t) \cong k[x_1,\ldots,x_t]/(x_1^3,\ldots,x_t^3)$, and 
$\mathscr{L}_{P_t}(\Gamma_t)$ has a $k$-basis $\{ x_1^{l_1}\cdots x_t^{l_t} : l_s \in \{ 1,2 \} \textrm{ for } 1 \le s \le t \}$.
We have a natural unimodular lattice triangulation $\Gamma_t$ of $P_t$ corresponding to the regular triangulation of Example~\ref{ex:counterexamples}.
The exact same computations as Example~\ref{ex:counterexamples} then show that 
the strong Lefschetz property fails for  $\mathscr{L}_{\theta_t}(P_t)$ if $0 < p \le t$, and the weak Lefschetz property fails for  $\mathscr{L}_{\theta_t}(P_t)$ if $p > 0$ and $2p - 1 \le t$.
\end{example}

\begin{figure}
\begin{center}
\begin{tikzpicture}[scale=4]
\draw (0.5,0.8) node[above] { $1$ } -- 
(1,0) node[right] { $3$ } -- 
(0,0) node[left] { $2$ } -- (0.5,0.8);

\draw (0.5,0) node[below] { $4$ };
\draw (0.5, 0) -- (0.5, 0.15);
\draw (0,0) -- (0.5, 0.15) node[above] {$5$} -- (1,0);
\draw  plot [smooth] coordinates {(0,0) (0.5,0.27) (1,0)};
\fill (0.5, 0.8) circle[radius = 0.5pt];
\fill (0.5, 0.15) circle[radius = 0.5pt];
\fill (1,0) circle[radius = 0.5pt];
\fill (0,0) circle[radius = 0.5pt];
\fill (0.5,0) circle[radius = 0.5pt];
\end{tikzpicture}
\end{center}
\caption{A triangulation that is not vertex-induced}\label{fig:noninduced}
\end{figure}
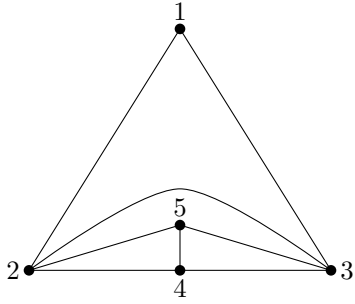

Finally, we give an example which shows that many of the results of this paper fail for quasi-geometric triangulations which are not vertex-induced. The authors of \cite{JKSS} use a cousin of Example~\ref{ex:nonvertexinduced} to construct examples, starting when $d=4$, of quasi-geometric triangulations where the local $h$-polynomial is not unimodal.

\begin{example}\label{ex:nonvertexinduced}
Consider the quasi-geometric triangulation $\Gamma$ of a $2$-dimensional simplex in Figure~\ref{fig:noninduced}, which appeared in \cite[Figure 1(c)]{JKSS}. Note that $\Gamma$ is not vertex-induced, as there are vertices $2$ and $3$ of excess $0$ such that $\{2, 3\}$ is interior. 
Fix a field $k$ 
and let $\theta$ be a special l.s.o.p. for $k[\Gamma]$. 
Then there is a surjection $k[x_4,x_5]/(x_4 \cdot x_5, x_5^2) \to A_\theta(\Gamma)$ which is an isomorphism in degrees at most $2$. 
In particular, $L_\theta(\Gamma) = k x_5 \oplus k x_2 x_3$ and $\ell(\Gamma; t) = t + t^2$. 
We see that $L^1(\Gamma)$ is contained in the socle of $L(\Gamma)$, and that weak Lefschetz fails for $L(\Gamma)$. 
\end{example}

\subsection{Relative local $h$-vectors}\label{ssec:relative}

We now describe the generalization of our results to relative local $h$-vectors. Let $\sigma \colon \Gamma \to 2^V$ be a topological triangulation of a simplex, and let $E$ be a face of $\Gamma$. For a face $F$ of a simplicial complex $\Delta$, let $\operatorname{lk}_F(\Delta)$ be the link of $F$ in $\Delta$, i.e., the collection of faces $G$ of $\Delta$ which are disjoint from $F$, and have the property that $F \sqcup G$ is a face of $\Delta$. The relative local $h$-polynomial is defined as
$$\ell(\Gamma, E; t) = \sum_{U \supset \sigma(E)} (-1)^{|V| - |U|} h(\operatorname{lk}_{\Gamma_U}(E); t).$$
The relative local $h$-vector is the vector of coefficients of the relative local $h$-polynomial. When $E = \emptyset$, the relative local $h$-vector is equal to $\ell(\Gamma)$. 

Like local $h$-vectors, the relative local $h$-vector $\ell(\Gamma, E) = (\ell_0, \dotsc, \ell_{d - |E|})$ of a triangulation is symmetric \cite[Remark 3.7]{AthanasiadisFlag}, with $\ell_s = \ell_{d - |E| - s}$. When $\Gamma$ is quasi-geometric, the relative local $h$-vector is nonnegative \cite{Athanasiadis12b,LPS1}. When $\Gamma$ is regular, the relative local $h$-vector is unimodal \cite{KatzStapledon16}.

\medskip

For a quasi-geometric triangulation $\sigma \colon \Gamma \to 2^V$ and a face $E$ of $\Gamma$, one can construct a module whose Hilbert function is $\ell(\Gamma, E)$. Let $k$ be a field, and let $\{ 1, \ldots, n \}$ be the vertex set of $\Gamma$. In this setting, we say that an l.s.o.p. $\theta = \{\theta_1, \dotsc, \theta_{d - |E|}\}$ for $k[\operatorname{lk}_{\Gamma}(E)]$ is \emph{special} if there is an injective function $f \colon V \setminus \sigma(E) \to \{1, \dotsc, d - |E|\}$ such that, for each vertex $i \in V \setminus \sigma(E)$, $\theta_{f(i)}$ is supported on vertices $j$ of $\operatorname{lk}_{\Gamma}(E)$ 
such that $i \in \sigma(j)$. Set
$$A_{\theta}(\Gamma, E) = k[\operatorname{lk}_{\Gamma}(E)]/(\theta_1, \dotsc, \theta_{d - |E|}),$$
and let $L_{\theta}(\Gamma, E)$  be the image of the ideal $(x^G : \sigma(G \sqcup E) = V)$ of $k[\operatorname{lk}_{\Gamma}(E)]$ in $A_{\theta}(\Gamma, E)$. Then $L_{\theta}(\Gamma, E)$ is a module over $k[\operatorname{lk}_{\Gamma}(E)]$, and for any special l.s.o.p. $\theta$, the Hilbert function of $L_{\theta}(\Gamma, E)$ is $\ell(\Gamma, E)$ \cite{Athanasiadis12b,LPS1}.

After renumbering, we may assume that $\sigma(E) = \{b+1, \dotsc, d\}$, where $b = d - |\sigma(E)|$. Set $K = k(a_{i,j})_{1 \le i \le d - |E|, \, 1 \le j \le n}$. For $1 \le i \le b$, set $\theta_i^{\operatorname{gen}} = \sum_{\sigma(j) \ni i} a_{i,j} x_j \in K[\operatorname{lk}_{\Gamma}(E)]$. For $b+1 \le i \le d - |E|$, set $\theta_i^{\operatorname{gen}} = \sum_{j} a_{i,j} x_j \in K[\operatorname{lk}_{\Gamma}(E)]$. Then $\theta^{\operatorname{gen}} = (\theta_1, \dotsc, \theta_{d - |E|})$ is a special l.s.o.p., and we set $L(\Gamma, E) = L_{\theta^{\operatorname{gen}}}(\Gamma, E)$. 

Using the differential operators technique, the proof of Theorem~\ref{thm:WL} can be easily modified to prove the following result. 

\begin{theorem}
Let $\sigma \colon \Gamma \to 2^V$ be a vertex-induced triangulation of a $(d-1)$-dimensional simplex. Let $E$ be a face of $\Gamma$, and assume that $d - |E| \ge 3$. If $k$ has characteristic $2$ or $0$, then multiplication by $\ell = 
\sum_{j \in \operatorname{lk}_{\Gamma}(E)} x_j$ induces an injection from $L^1(\Gamma, E)$ to $L^2(\Gamma, E)$. 
\end{theorem}

The proof of Theorem~\ref{thm:SL} does not immediately generalize to the relative setting. The proof of Lemma~\ref{lem:generation} used that, for each $i \in V$, there is a unique vertex $j$ for which $x_j$ only appears in $\theta_i$. This can fail if $E$ is nonempty: there can be multiple $j$ such that $x_j$ only appears in some $\theta_i$. Indeed, we do not know if $L(\Gamma, E)$ has the strong Lefschetz property when $d - |E| = 4$. 

\medskip

As in the case $E = \emptyset$, if $\Gamma$ is a regular triangulation, then the relative local $h$-polynomial has an interpretation in terms of intersection cohomology \cite[Theorem 6.1]{KatzStapledon16}. As in Section~\ref{sec:char0}, let $f \colon X_{\Gamma} \to \mathbb{A}^d$ be the projective toric morphism corresponding to the triangulation. The face $E$ defines a torus-orbit closure $V(E)$ on $X_{\Gamma}$. Then $\ell(\Gamma, E)$ enumerates the copies of shifts of $\underline{\mathbb{Q}}_{o}$ appearing in a decomposition of $Rf_* \underline{\mathbb{Q}}_{V(E)}$. 

There is a natural special l.s.o.p. $\theta$ on $\mathbb{Q}[\operatorname{lk}_{\Gamma}(E)]$, which is defined in terms of the fan of $X_{\Gamma}$, for which $H^*(V(E)) = A_{\theta}(\operatorname{lk}_{\Gamma}(E))$ (after halving the degrees in $H^*(V(E))$). As in the proof of Theorem~\ref{thm:regular}, $L_{\theta}(\Gamma, E)$ can be identified with a summand of the associated graded of $H^*(V(E))$ with respect to the perverse filtration, and the relative Hard Lefschetz theorem then implies the following result.

\begin{theorem}
Let $\sigma \colon \Gamma \to 2^V$ be a regular triangulation of a simplex, and let $E$ be a face of $\Gamma$. If $k$ has characteristic $0$, then $L(\Gamma, E)$ has the strong Lefschetz property. 
\end{theorem}

\subsection{Generators of local face modules}\label{ssec:generation}

In this section, we make a couple of conjectures related to generating local face modules. 
Let $\sigma \colon \Gamma \to 2^V$ be a quasi-geometric triangulation of a $(d-1)$-dimensional simplex, and let $k$ be a field. 
A consequence of the strong Lefschetz property for $L(\Gamma)$, if it holds, is that $L(\Gamma)$ is generated as a $K[\Gamma]$-module in degree at most $d/2$. While we have seen that the strong Lefschetz property can fail if the characteristic of $k$ is positive, the authors are not aware of any vertex-induced examples where this generation property fails. 

\begin{conjecture}\label{conj:generation}
Let $\sigma \colon \Gamma \to 2^V$ be a vertex-induced triangulation of a $(d-1)$-dimensional simplex, and let $k$ be a field. Then $L(\Gamma)$ is generated as a $K[\Gamma]$-module in degree at most $d/2$. 
\end{conjecture}

Conjecture~\ref{conj:generation} is closely related to understanding the socle of local face modules. Let 
$$\operatorname{Soc} L^s(\Gamma) = \{u \in L^s(\Gamma) : x_j \cdot u = 0 \text{ for all vertices }j \in \Gamma\}.$$
\begin{proposition}
For each $s$, the dual of $\operatorname{Soc} L^s(\Gamma)$ is identified with $L^{d - s}(\Gamma)/K[\Gamma]^1 \cdot L^{d - s - 1}(\Gamma)$. 
\end{proposition}

\begin{proof}
Given $u \in \operatorname{Soc} L^s(\Gamma)$ and $v \in L^{d-s-1}(\Gamma)$, for any vertex $j$ of $\Gamma$ we have
$$B(u, x_j \cdot v) = B(x_j \cdot u, v) = B(0, v) = 0.$$
Conversely, if $u \in L^s(\Gamma)$ has $B(u, x_j \cdot v) = B(x_j \cdot u, v) = 0$ for all $v \in L^{d-s-1}(\Gamma)$,  then $x_j \cdot u = 0$ by the nondegeneracy of $B$. In particular, if  $B(u, x_j \cdot v)  = 0$ for all $v \in L^{d-s-1}(\Gamma)$ and vertices $j$ of $\Gamma$, then $u \in \operatorname{Soc} L^{s}(\Gamma)$. 
\end{proof}

In particular, Conjecture~\ref{conj:generation} is equivalent to showing that $\operatorname{Soc}L^s(\Gamma) = 0$ for $s < d/2$. 

Suppose that $d$ is even. As noted in Example~\ref{ex:counterexamples}, the symmetric bilinear form $B$ on $L^{d/2}(\Gamma)$ need not be anisotropic, even if $\Gamma$ is regular and $k = \mathbb{Q}$.

\begin{conjecture}\label{conj:socleanisotropy}
Let $\sigma \colon \Gamma \to 2^V$ be a vertex-induced triangulation of a $(d-1)$-dimensional simplex, and let $k$ be a field. If $d$ is even, then the restriction of $B$ to $\operatorname{Soc} L^{d/2}(\Gamma)$ is anisotropic. 
\end{conjecture}

We conclude with two classes of examples where Conjecture~\ref{conj:socleanisotropy} holds. 

\begin{example}
Suppose that $\Gamma$ is regular and $k = \mathbb{Q}$. As in the proof of Theorem~\ref{thm:regular}, let $\theta$ be the special l.s.o.p. for which $L_{\theta}(\Gamma)$ is identified with a summand of the perverse filtration on $H^*(X_{\Gamma})$, and let $\eta \in \mathbb{Q}[\Gamma]^1$ be a relatively ample class. The relative Hodge--Riemann relations \cite[Theorem 3.3.1(7)]{dCM} imply that the restriction of $B$ to the kernel of the map $L^{d/2}_{\theta}(\Gamma) \to L^{d/2 +1}_{\theta}(\Gamma)$ given by multiplication by $\eta$ is $(-1)^{d/2}$-definite. In particular, it is anisotropic. 

This implies that the restriction of $B$ to the kernel of the map $L^{d/2}(\Gamma) \to L^{d/2 +1}(\Gamma)$ given by multiplication by $\eta$ is anisotropic. Indeed, by specializing the coefficients of $\theta^{\operatorname{gen}}$ to $\theta$ one at a time, we can realize the restriction of $B$ to the kernel in $L^{d/2}_{\theta}(\Gamma)$ as a specialization of the restriction of $B$ to the kernel in $L^{d/2}(\Gamma)$. The claim follows, because a quadratic form which specializes to an anisotropic quadratic form is anisotropic. As $\operatorname{Soc} L^{d/2}(\Gamma)$ is contained in the kernel of multiplication by $\eta$, this implies Conjecture~\ref{conj:socleanisotropy} in this case. 
\end{example}

\begin{example}
Suppose that $\Gamma$ is vertex-induced and semi-small, as in Example~\ref{ex:semismall}. Then $L(\Gamma)$ is concentrated in degree $d/2$, so $\operatorname{Soc} L^{d/2}(\Gamma) = L^{d/2}(\Gamma)$. Suppose that $k$ has characteristic $2$. 
Recall that  $L^{d/2}(\Gamma)$ has a basis given by $\{x^F : |F| = d/2, \, \sigma(F) = V\}$. Fix an interior face $F$ with  $|F| = d/2$. 
We claim that $|\sigma(G)| = 2 |G|$ for any $G \subset F$. Assuming this claim, the hypothesis of Lemma~\ref{lem:generationmiddle} holds, and so the restriction of $B$ to $\operatorname{Soc} L^{d/2}(\Gamma)$ is anisotropic. The same holds when $k$ has characteristic $0$, by an argument similar to the proof of Theorem~\ref{thm:SL}. 

It remains to verify the claim. 
The semi-small condition implies that $|\sigma(G)| \le 2 |G|$ and $|\sigma(F \smallsetminus G)| \le 2 |F \smallsetminus G| = d - 2|G|$.
Adding these two inequalities gives $|\sigma(G)| + |\sigma(F \smallsetminus G)| \le d$. 
Since $\Gamma$ is vertex-induced and $\sigma(j) \subset \sigma(G) \cup \sigma(F \smallsetminus G)$ for all vertices $j$ of $F$, we deduce that 
$\sigma(G) \cup \sigma(F \smallsetminus G) = \sigma(F) = V$, and hence all the above inequalities are actually equalities. In particular, $|\sigma(G)| = 2 |G|$, as desired.

\end{example}

\bibstyle{amsalpha}
\bibliography{localface.bib}

\end{document}